\documentclass[11pt]{article}

\usepackage{layout}
\usepackage{amsmath, amsthm}
\usepackage{amssymb, amsrefs}
\usepackage{stmaryrd}
\usepackage{graphicx}
\usepackage{setspace}
\usepackage{textcomp}
\usepackage{esvect}
\usepackage[letterpaper]{geometry}
\usepackage{arydshln}
\usepackage[all]{xy}
\usepackage{etex}

\usepackage{epsfig}
\usepackage{latexsym}
\usepackage{amsthm}

\usepackage{mathrsfs}

\usepackage[colorlinks,citecolor=blue]{hyperref}

\usepackage[latin1]{inputenc}

\usepackage{tikz-cd}
\usepackage{pgfplots}

\numberwithin{equation}{section} 

\newtheorem{theorem}[equation]{Theorem}

\newtheorem{proposition}[equation]{Proposition}
\newtheorem{corollary}[equation]{Corollary}
\newtheorem{problem}[equation]{Problem}

\theoremstyle{definition}
\newtheorem{definition}[equation]{Definition}
\newtheorem{example}[equation]{Example}

\newcommand{\R}[1]{\mathbb{R}^{#1}}

\newcommand{\Z}[1]{\mathbb{Z}^{#1}}
\newcommand{\embed}[0]{\hookrightarrow}

\newcommand{\T}{\rotatebox[origin=c]{180}{$\scriptscriptstyle \perp $}}

\newcommand{\dsum}{\displaystyle\sum}

\begin{document}
\title{\sffamily Generalized Affine Programming $\&$ Duality Gap with non-Division Rings}
\author{ \textsc{Tien Chih} \\[0.25em]}
\date{\today}
\maketitle

\begin{abstract}
Classical primal-dual affine programming takes place over finite dimensional real vector spaces.  This results in beautiful duality theory, connecting the optimal solutions of the primal maximization problem and the dual minimization problems.  These results include the Existence Duality Theorem \cite{Tucker}, which guarantees optimal solutions to any feasible bounded program;  and the Strong Duality Theorem \cite{Tucker}, which implies that optimal solutions for primal and dual programs must have the same objective value.  In a common extension of classical affine programming, we see that the Strong Duality does not hold when ring of scalars is the integers.  Extension of classical affine programming results to ordered division rings are explored in \cite{Bartl}.  In this paper, we describe the generalized setting of affine programming using ordered ring (not necessarily division), and classify the rings for which the Existence Duality Theorem or the Strong Duality Theorem fail.
\end{abstract}

\section{Classical and Integer Affine Programming}

In classical primal-dual affine programming one is asked to solve the following problems:

\begin{problem}[\cite{Tucker}]\label{Classic}
Given $A\in \R{n\times m}, \vec{b}\in \R{m}, \vec{c}\in \R{n}$, find $\vec{x}\in \R{n}$ subject to:  
\begin{eqnarray*}
A\vec{x}&\leq&\vec{b} \\ 
\vec{x}&\geq&0  
\end{eqnarray*}
so that  $f(\vec{x})=\vec{c}^{\T}\vec{x}$ is maximized.  This is called the \textit{primal maximization program}.

Also, given the same initial data $(A, \vec{b}, \vec{c}, d)$ we can ask to find $\vec{y}\in \R{n}$ subject to:
\begin{eqnarray*}
\vec{y}^{\T}A &\geq& \vec{c}^{\T} \\
\vec{y} &\geq&0  
\end{eqnarray*}

so that  $g(\vec{y})=\vec{y}^{\T}\vec{b}$ is minimized.  This is called the \textit{dual minimization program}.

\end{problem}

This setting leads to some powerful and useful duality results connecting the solutions to the primal and dual problems.  One of them is the \textit{Weak Duality Theorem:}

\begin{theorem}[Weak Duality Theorem \cite{Tucker}]\label{introWD}
Given $A\in \R{m\times n}, \vec{b}\in \R{m}, \vec{c}\in \R{n}, d\in R$, and $f,g$ described in Problem \ref{Classic}, and that $\R{}$ is an ordered ring of scalars), we have that: $$g(\vec{y})\geq f(\vec{x})$$ for feasible $\vec{x}, \vec{y}$.
\end{theorem}

In other words, equality between feasible solutions implies that both solutions are optimal.  It is also the case that in this classical setting, the converse is true as well:

\begin{theorem}[Strong Duality Theorem \cite{Tucker}]\label{introSD}
Let $\vec{x}^*\in \R{n}, \vec{y}^*\in \R{m}$ be a pair of feasible optimal solutions for the primal and dual programs respectively.  Then $$f(\vec{x}^*)=g(\vec{y}^*).$$
\end{theorem}

The last classical theorem we will discuss in this paper is a classification result, describing the possible types of solutions for primal-dual programs, the \textit{Existence-Duality Theorem}:

\begin{theorem}[Existence-Duality Theorem \cite{Tucker}]\label{EDT}
Given a primal-dual affine program, exactly one of the following hold:
\begin{enumerate}
\item Both the primal and dual programs are infeasible.
\item The primal program is infeasible and the dual program is unbounded.
\item The dual program is infeasible and the primal program is unbounded.
\item Both the primal and dual program achieve an optimal solution.
\end{enumerate}
\end{theorem}
Another way of interpreting this theorem is as this following collection of statements:
\begin{corollary}[Existence-Duality Theorem]\label{corEDT}
Given a primal-dual affine program, then the following hold:
\begin{enumerate}
\item If both programs are infeasible then there are no statements about the boundedness or optimality of either program.
\item If both programs are feasible and the primal program is infeasible, then the dual program is unbounded.
\item If both programs are feasible and the dual program is infeasible, then the primal program is unbounded
\item If either program is feasible and bounded, then both programs are feasible and bounded and also obtain an optimal solution.
\end{enumerate}
\end{corollary}

A well studied variant of this classical programming is \textit{integer programming} where $\Z{}$ replaces $\R{}$ as the underlying ring of scalars.   When we do this, we see that both the Strong Duality Theorem and $(2), (3)$ and the Existence-Duality Theorem fail to extend this setting.

\begin{example}[Counterexample to Strong Duality]\label{ceSD}
Let $A\in \Z{1\times 1}, A:=\begin{pmatrix} 2 \end{pmatrix}, \vec{b}\in \Z{1}, \vec{b}:=1, \vec{c}\in \Z{1}, \vec{c}:=1, d\in \Z{}, d:=0$.  Our primal program is finding $x\in \Z{}$ subject to:  $$2x<1, x>0 $$ that maximizes $f(x)=x$.  The dual program is finding $y\in \Z{}$ subject to $$2y>1, y>0$$ that minimizes $g(y)=y$.  We can verify that our optimal solutions are  $x=0, y=1$.  However $f(0)=0<1=g(1)$, and so the Strong Duality Theorem does not hold for $\Z{}$.

\end{example}

\begin{example}[Counterexample to 2,3 of Existence-Duality]
Let $A\in \Z{2\times 1}, A:=\begin{pmatrix} 2 \\ -2 \end{pmatrix}, \vec{b}\in \Z{2}, \vec{b}:=:=\begin{pmatrix} 1 \\ -1 \end{pmatrix}, \vec{c}\in \Z{1}, \vec{c}:=0, d\in \Z{}, d:=0$.  Our primal program is finding $x\in \Z{}$ such that $$\begin{pmatrix} 2 \\ -2 \end{pmatrix}x \leq \begin{pmatrix} 1 \\ -1 \end{pmatrix}, x\geq 0.$$  The dual program is finding $\vec{y}\in \Z{2}$ so that $$\vec{y}^{\T}\begin{pmatrix} 2 \\ -2 \end{pmatrix}\geq 0, \vec{y}\geq 0.$$ We notice that the primal program is only feasible if $2x\leq 1, -2x\leq -1$, which is only possible if $2x=1$, which cannot happen if $x\in \Z{}$.  Thus the primal program is infeasible.  However, the dual program has an optimal solution $\vec{y}= \begin{pmatrix} 0 \\ 0 \end{pmatrix}$ and the dual program is not unbounded (or infeasible).  We can take the transpose of $A$ and swap $\vec{b}, \vec{c}$ to find a program where the dual is infeasible and the primal achieves optimality.
\end{example}

Both of these example exploit the fact that $2$ is not invertible in $\Z{}$.  Thus it is natural to ask: ``If we extend affine programming to scalars over any ordered non-division ring, can we be guaranteed that similar counterexamples will exist?"

\section{Generalized Affine Programming}

When we generalize the setting of affine optimization, one of the possible generalizations is the rings from which we pick our scalars.  In order to reasonably discuss optimization and optimality, one should require that these rings are \textit{ordered} in a reasonable way (i.e.\ not just well-ordered with the axiom of choice).  

\begin{definition}[Ordered Ring \cite{Lam}]
A ring $R$ is \textit{ordered} (sometimes called \textit{trichotomy ordered}), if there is a non-empty subset $P\subset R$ called the ``positives" with the following properties:
\begin{enumerate}
\item $R$ can be partitioned into the disjoint union: $P\sqcup \{0\}\sqcup -P$ (this is called the trichotomy property).  
\item Given $a,b\in P$, then $a+b\in P$.  
\item Given $a,b\in P$, then $ab\in P$.  
\end{enumerate}
We say that $a\geq b$ if $a-b\in P\cup \{0\}$.    We also say that $a>b$ if $a-b\in P$.    
\end{definition}

So rings such as $\Z{}, \mathbb{Q}, \R{}$ are all common examples of ordered rings.  However, we can construct some examples of more general ordered rings.  

\begin{example}
Let $\R{}[x]$ be the ring of polynomials over $\R{}$, and let $$P:=\{p(x):p(x)\ \text{has a positive leading coefficient}\}.  $$  We first notice that polynomials with positive leading coefficients are closed under sums and products.    Moreover, each polynomial either has a positive leading coefficient, or a negative leading coefficient, or is 0.    Thus, $P$ satisfies the conditions of being a set of positives.  
\end{example}

This example may be extended to a \textit{non-commutative} example of an ordered ring.

\begin{example}
Consider the ring $R:=\R{}\langle x,y \rangle/\langle yx-2xy\rangle$, i.e.\   the ring of polynomials over $x,y$ with the skew-product $yx=2xy$.      We can establish a ``lexicographical" ordering of the monomials of $R$, with $y^{n_1}x^{m_1}>y^{n_2}x^{m_2}$ if $n_1>n_2$ or $n_1=n_2, m_1>m_2$.    Then, we can define the positives $P$ of $R$ to be the  polynomials with positive leading coefficient, where leading means greatest with respect to the above lexicographical ordering.  

One can easily see that $P$ satisfies the properties of a collection of positives, and $R$ is non-commutative.  
\end{example}

Letting $R$ be an ordered ring, we can state a more general version of a primal-dual affine programming problem:

\begin{problem}[\cite{Chih}]\label{General}
Given $A\in R^{n\times m}, \vec{b}\in R^{m}, \vec{c}\in R^{n}$, find $\vec{x}\in R^{n}$ and $\vec{y}\in R^{m}$ subject to:  
\begin{eqnarray*}
A\vec{x}&\leq&\vec{b} \\ 
\vec{y}^{\T}A &\geq& \vec{c}^{\T} \\
\vec{x}, \vec{y}&\geq&0  
\end{eqnarray*}
so that  $f(\vec{x})=\vec{c}^{\T}\vec{x}$ is maximized and $g(\vec{y})=\vec{y}^{\T}\vec{b}$ is minimized. These are the primal and dual programs.
\end{problem}

The existence of non-commutative ordered rings complicates the discussion of duality gap.  Let $R$ be a non-commutative ordered division ring and let $a\in R$ be a non-unit, non-central element of $R$.  Then , consider the primal-dual program where $A=a, b=c=1, d=0$, similar to Example \ref{ceSD}.  It is possible that one may find $x, y$ so that $ax<1, ya>1$ and yet $x=y$.  Thus using the non-invertibility of $R$ may not be sufficient to exhibit a duality gap between the primal and dual solutions.  Further investigation into the structure of ordered rings is necessary.

\section{Exhibiting Duality Gap}
We first verify that the Weak Duality still extends to this setting.  We begin with a preliminary discussion:

\begin{definition}
Given $A, \vec{b}, \vec{c}, d, \vec{x}, \vec{y}$ as in Problem \ref{General}, we may define \textit{slack variables}, primal slack variable  $\vec{t}:=\vec{b}-A\vec{x}$, and dual slack variable $\vec{s}:=\vec{y}^{\T}A-\vec{c}$.
\end{definition}

Notice that if $\vec{x}$ satisfies $A\vec{x}\leq \vec{b}$, then $\vec{t}\geq 0$.  Similarly if $\vec{y}^{\T}A-\vec{c}\geq 0, \vec{s}\geq 0$.

\begin{proposition}[\cite{Chih}]\label{key}
Given $A, \vec{b}, \vec{c}, d, \vec{x}, \vec{y}$ as in Problem \ref{General}, we have that:  $$\vec{x}^{\T}\vec{s}+(-1)g(\vec{y})=\vec{y}^{\T}(-\vec{t})+(-1)f(\vec{x}).$$
\end{proposition}
\begin{proof}
We note that:
\begin{eqnarray*}
\vec{x}^{\T}\vec{s}+(-1)g(\vec{y})&=&\vec{x}^{\T}(A\vec{y}-\vec{c})-\vec{b}^{\T}\vec{y}+d\\
&=&\vec{x}^{\T}A\vec{y}-\vec{x}^{\T}\vec{c}-\vec{b}^{\T}\vec{y}+d.\\
&=&\vec{y}^{\T}(A\vec{x}-\vec{b})-\vec{c}^{\T}\vec{x}+d\\
&=&\vec{y}^{\T}(-\vec{t})+(-1)f(\vec{x})
\end{eqnarray*}

\end{proof}

This is an extension of ``Tucker's Key Equation".  As a corollary, we have ``Tucker's Duality Equation":

\begin{corollary}\label{tuckerduality}
Given $A, \vec{b}, \vec{c}, d, \vec{x}, \vec{y}$ as in Problem \ref{General}, we have that:
$$g(\vec{y})-f(\vec{x})=\vec{s}^{\T}\vec{x}+\vec{y}^{\T}\vec{t}.$$
\end{corollary}

Using this equation, we may extend Weak Duality to this setting:

\begin{theorem}[Weak Duality \cite{Chih}] \label{WD}
Given $A, \vec{b}, \vec{c}, d, \vec{x}, \vec{y}$ as in Problem \ref{General}, we have that: $$g(\vec{y})\geq f(\vec{x})$$ for feasible $\vec{x}, \vec{y}$.
\end{theorem}
\begin{proof}
Using that $R$ is an ordered ring of scalars; and since $\vec{x}, \vec{y}$ are feasible, each entry $\vec{x}_i, \vec{y}_j\geq 0$ for $1\leq i\leq n, 1\leq j \leq m$.  Moreover, the same is true for $\vec{s}_i, \vec{t}_j$.  Thus $$\vec{s}^{\T}\vec{x}=\dsum_{i=1}^n \vec{s}_i\vec{x}_i\geq 0$$ and $$ \vec{y}^{\T}\vec{t}=\dsum_{j=1}^m \vec{y}_j\vec{t}_j \geq 0.$$
So it follows that:
\begin{eqnarray*}
g(\vec{y})-f(\vec{x})&=&\vec{s}^{\T}\vec{x}+\vec{y}^{\T}\vec{t}\ \geq\ 0\\
g(\vec{y})&\geq &f(\vec{x}).
\end{eqnarray*}
\end{proof}

We may then conclude that in this setting, the solutions to the primal problems will always be bounded above by the solutions to the dual problems and vice versa.  However, when $R$ is not division, we would like to conclude that there are programs where this is a strict bound and equality is not possible.  To show this, we further examine the properties of $R$.

\begin{proposition}[\cite{Chih}]\label{propnocenter}
Let $R$ be an ordered unital ring with set of positives $P$.  Let $a,b\in P$, such that, without loss of generality,  $ab\leq ba$.    Then there is no element $z\in Z(R)$ such that $ab<z<ba$.  
\end{proposition}
\begin{proof}
Suppose that this were not true, then $aba<za=az<aba$, a contradiction.  
\end{proof}

\begin{corollary}
If $R$ is a unital ring, and there are $a,b\in P$ such that $ab$ is finite and $ba\geq ab$, then $ba=ab+\epsilon$, where $\epsilon$ is zero or infinitesimal.  
\end{corollary}
\begin{proof}
We proceed via contradiction, and assume that $ba$ is infinite.  Since $ab$ is finite, there is an $m\in \Z{}_+$ such $ab<m<ba$, which contradicts Proposition \ref{propnocenter}, (recall that $Z\embed Z(R))$.    Thus $ba-ab$ is finite.    If this difference is not infinitesimal or zero, then there is a $n\in \Z{}_+$  such that $nba-nab>1$.   Since $nab$ is finite, we may find an integer $m_1\leq nab$ and moreover, we may find a maximum such integer (else $nba$ would be infinite.) 

There is an integer $z$ in between $nab<z<nba$, since,  $m_1+1>nab$ but $m_1+1<nba$.  Since $z$ is an integer, it is central, and by Proposition \ref{propnocenter}, this is a contradiction.  
\end{proof}

Proposition \ref{propnocenter} allows us to show the existence of a duality gap for all non-division ordered rings.

\begin{theorem}[\cite{Chih}]\label{propcounter}
Let $R$ be an ordered ring, but not a division ring.    Then there is a primal-dual program such that both problems are feasible, and $f(\vec{x})<g(\vec{y})$ for all feasible choices of $\vec{x}, \vec{y}$.  
\end{theorem}
\begin{proof}
Let $n, m=1$, where $n, m$ are as in Problem \ref{General}. Then let $A=a$, a positive non-invertible element of $R$.    Let $\vec{b}=1$, $\vec{c}=1$, $d=0$.    The primal program then, is to maximize $x\in R$ subject to $xa<1$.    Conversely, the dual program is to minimize $y\in R$ subject to $ay>1$.    For any pair of feasible solutions $x,y\in R$, we have that $xa<1<ay$, and so by Proposition \ref{propnocenter} $x\neq y$, and thus $x<y$.  

\end{proof}

We can then classify the programs where the Strong Duality fails.

\begin{theorem}\label{noSD}
Given $R$ an ordered non-division ring, if there is a smallest positive element $p$, then we may find a primal-dual program where the primal and dual programs obtain optimal solutions, but they do not share the same objective value.
\end{theorem}
\begin{proof}
We first notice that if $R$ has a smallest positive value, then it is 1.  To see this, suppose that there is a $a\in R, 0<a<1$.  For any $n\in \Z{}, n\geq 0$ we can show that $a^n-a^{n+1}>0$.  Consider that $a^n-a^{n+1}=a^n(1-a)$, since $a, 1-a>0$ this is a product of positive elements and thus positive.  This contradicts there been a smallest positive $p$, and thus no such $a$ exists.

We then let $n, m=1$, where $n, m$ are as in Problem \ref{General}. Then let $A=a$, a positive non-invertible element of $R$.    Let $\vec{b}=1$, $\vec{c}=1$, $d=0$.   Since there is no solution to $0<xa<1$ then $xa>1$ for any positive $x$.  Thus the only feasible primal solution is $x=0$.  Conversely we notice that $a\neq 1$ since $a$ is not a unit.  Thus $ay>1$ for any positive $y$.  Clearly $y=0$ is not feasible, and so the smallest value for $y$ is $y=1$.  Both programs admit optimal solutions, but $f(0)=0, g(1)=1$ and they do not share the same objective value.
\end{proof}

Not only does the Strong Duality theorem fail to extend to programs over non-division rings, but the Existence-Duality theorem also fails to extend.  For programs over any non-division ring, we see that (2), (3) of the Existence-Duality Theorem fail to extend.  We will also see that for rings where 1 is not the smallest positive (i.e.\  the negation of the hypothesis for which the Strong Duality fails), (4) of the Existence-Duality theorem fails instead.

\begin{theorem}
Let $R$ be a non-division ordered ring.  Then there is a primal-dual program over $R$ so that the primal (dual) program is infeasible, but the dual (primal) program is neither infeasible or unbounded.
\end{theorem}
\begin{proof}
Let $a\in R$ be a non-unit.  Then let $A\in R^{2\times 1}, A:=\begin{pmatrix} a \\ -a \end{pmatrix}, \vec{b}\in R^{2}, \vec{b}:=\begin{pmatrix} 1 \\ -1 \end{pmatrix}, \vec{c}\in R^{1}, \vec{c}:=0, d\in \Z{}, d:=0$.  The primal program asks us to find $x$ such that $ax\leq 1, -ax\leq -1$, which only occurs when $ax=1$, thus the primal program is infeasible, but the dual program is optimized at $\vec{y}=\begin{pmatrix} 0 \\ 0 \end{pmatrix}$. So clearly the dual program is bounded and feasible.  We can take the transpose of $A$ and swap $\vec{b}, \vec{c}$ to reverse the roles of the primal and dual problems.
\end{proof}

Moreover, in Theorem \ref{noSD} we saw that when $1$ is the smallest positive, the Strong Duality theorem fails to extend.  We will now show that of $1$ is not the smallest positive, then (4) of the Existence-Duality Theorem fails to extend.

\begin{theorem}
Let $R$ be an ordered non-division ring, and suppose $1$ is not the smallest positive.  Then we may can find a primal (or dual) program that is feasible and bounded but does not achieve optimality.
\end{theorem}
\begin{proof}
Let $a$ be a positive non-unit.  Then there are two cases.

If there is a solution to $0<xa<1$, say $x=z$ then let $A=a, b=1, c=1, d=0$.  Since $za<1$, we have that $(za)^{n+1}<(za)^n$ and thus $f((za)^{n+1}z)=-(za)^{n+1}z>-(za)^{n}z=f((za)^{n}z)$.  Thus the primal program is feasible, and bounded above by 0, but does not obtain an optimal solution.

If there is no solution to $0<xa<1$, then t $A=a, b=1, c=1, d=0$.  Notice that for any $y>0$, $ay>1$.  Since $1$ is not the smallest positive, there is an element $p\in R, 0<p<1$.  We notice that given any feasible $y, yp<y$, but $ayp>1$.  Thus the dual program is feasible and bounded below by 1, but does not obtain an optimal solution.

\end{proof}

\section{Conclusion}

It is natural to extend the idea of affine programming to a more general setting, this extension sometimes includes the ring of scalars over which we do our optimization.  We hope that in doing so, the results which hold and fail are consistent with their natural analogues, that is Strong Duality Theorem and the Existence-Duality Theorem hold for division rings, but not for non-division rings.  While the duality gap is something that is intuitive and simple over the integers, we see that general ordered rings are far more complex.  In particular, not all non-division ordered rings have a smallest element, or even a commutative product.  However, despite these potential complications, one may find general arguments to show these results failing for non-division rings.  In \cite{Chih}, it is shown that these results to extend to ordered division rings, and finite dimensional vector-spaces over these rings.  Thus these results hold and fail in their general, possibly non-commutative analogues.

\section*{Acknowledgements}
I'd like to thank Nikolaus Vonessen for his time and help.   I would especially like to thank my advisor George McRae for his insight efforts and perspective.

\addcontentsline{toc}{section}{References}
\bibliographystyle{alpha}
\bibliography{Research}
\noindent \textsc{Department of Sciences and Mathematics\\ Newberry College\\Newberry, SC 29108, USA} \\
\textit{E-mail}: \href{mailto:tien.chih@newberry.edu}{tien.chih@newberry.edu}\\


\end{document}